\documentclass[reqno,12pt,epsf]{amsart}
\usepackage{amsmath}
\usepackage{latexsym,amsmath}
\usepackage{amsthm}
\usepackage{amssymb}
\usepackage{graphics}

\newtheorem{theorem}{Theorem}     

\newtheorem{conjecture}[theorem]{Conjecture}

\def\om{\omm}
\def\omm{\rho}

\def\R{\mathbb{R}}
\def\N{\mathbb{N}}

\newcounter{rot}

\title[The 1-2-3 Conjecture for Hypergraphs]
{The 1-2-3 Conjecture for Hypergraphs}

\author[M. Kalkowski]{Maciej Kalkowski}
\curraddr[M. Kalkowski]{Adam Mickiewicz University\\ Faculty of Mathematics and
 Computer Science\\ Pozna\'n, Poland}
\email{kalkos@atos.wmid.amu.edu.pl}
\author[M. Karo\'nski]{Micha{\l} Karo\'nski}
\curraddr[M. Karo\'nski]{Adam Mickiewicz University\\ Faculty of Mathematics and
 Computer Science\\ Pozna\'n, Poland\\ and Emory University\\
Department of Mathematics and Computer Science\\
Atlanta, GA, USA}
\email{karonski@amu.edu.pl}
\author[F. Pfender]{Florian Pfender}
\curraddr[F. Pfender]{University of Colorado Denver\\ Department of Mathematical and Statistical Sciences\\ Denver, CO, USA}
\email{Florian.Pfender@ucdenver.edu}
\thanks{Research of the third author supported in part by a Simons Foundation Collaboration Grant}

\begin{document}

\begin{abstract}
A weighting of the edges of a hypergraph is called vertex-coloring if the weighted degrees of the vertices
yield a proper coloring of the graph, i.e., every edge contains at least two vertices with different weighted degrees. In this paper we show that such a weighting is possible
from the weight set $\{ 1,2 ,\ldots,\max\{5,r+1\}\}$ for all hypergraphs with maximum edge size $r\ge 3$ and not containing edges solely consisting of identical vertices. The number $r+1$ is best possible for this statement.
\end{abstract}

\maketitle

\section{Introduction and Notation}
Regular graphs have been studied in a lot of contexts, and have many properties not shared by other graphs. One may ask what is on the other side of the spectrum, and look for graphs which are as irregular as possible. But what is irregular? It is an easy observation that every graph with at least two vertices contains a pair of vertices of equal degree, so one can not hope for graphs which are {\em totally irregular} in the sense that all vertices have pairwise different degrees. This changes if one considers multigraphs. In fact, by multiplying some edges, one can make every graph totally irregular, as long as the original graph does not contain an isolated edge or two isolated vertices. This observation led to the definition of the {\em irregularity strength} of a graph in~\cite{CJLORS}, the minimum maximum multiplicity one has to use on a given graph.

Later, Karo\'{n}ski, {\L}uczak and Thomason~\cite{KLT} asked a similar question inspired by this concept. What if we do not require {\em all} vertices to have pairwise different degrees, but only require this difference for {\em adjacent} vertices? In other words, we want that the degrees yield a proper vertex coloring. This question led to the so called 1-2-3-Conjecture, stated here in the obviously analogous form using edge weights instead of multiplicities.
\begin{conjecture}\label{123}
For every graph $G$ without isolated edges, there is a weighting $\omm: E(G)\to \{1,2,3\}$,
such that the induced vertex weights $\omm(v):=\sum_{e \ni v}\omm(e)$ properly color $V(G)$.
\end{conjecture}
The 1-2-3-Conjecture is known to be true for several classes of graphs, the best known result for general graphs is by the authors of the current article~\cite{KKP}.
\begin{theorem}\label{125}
For every graph $G$ without isolated edges, there is a weighting $\omm: E(G)\to \{1,2,3,4,5\}$,
such that the induced vertex weights $\omm(v):=\sum_{e \ni v}\omm(e)$ properly color $V(G)$.
\end{theorem}
Shortly thereafter, a total version of the 1-2-3-Conjecture, adaptly called the 1-2-Conjecture, was formulated by Przyby{\l}o and Wo\'{z}niak~\cite{PW}.
\begin{conjecture}\label{12con}
For every graph $G$, there is a weighting $\omm: E(G)\cup V(G)\to \{1,2\}$,
such that the induced total vertex weights $w(v):=\omm(v)+\sum_{e \ni v}\omm(e)$ properly color $V(G)$.
\end{conjecture}
Kalkowski in~\cite{K} came close to settling this conjecture.
\begin{theorem}\label{tot123}
For every graph $G$, there are weightings $\omm: E(G)\to \{1,2,3\}$ and $\omm': V(G)\to \{1,2\}$
such that the induced total vertex weights $w(v):=\omm'(v)+\sum_{e \ni v}\omm(e)$ properly color $V(G)$.
\end{theorem}
One natural and promising approach for both conjectures is the use of Alon's Combinatorial Nullstellensatz (see~\cite{A}). In its most straightforward application, it would prove list versions of the conjectures if successful, leading to the following stronger conjectures, first stated by Bartnicki, Grytczuk and Niwczyk, and by Przyby{\l}o and Wo\'{z}niak, and Wong and Zhu, respectively.
\begin{conjecture}\label{123list}\cite{BGN}
For every graph $G$ without isolated edges, and for every assignment of lists of size $3$ to the edges of $G$, there exists a weighting $\omm: E(G)\to \R$ from the lists,
such that the induced vertex weights $\omm(v):=\sum_{e \ni v}\omm(e)$ properly color $V(G)$.
\end{conjecture}
\begin{conjecture}\label{12list}\cite{PW2},\cite{WZ}
For every graph $G$, and for every assignment of lists of size $2$ to the vertices and edges of $G$, there exists a weighting $\omm: V(G)\cup E(G)\to \R$ from the lists,
such that the induced total vertex weights $w(v):=\omm(v)+\sum_{e \ni v}\omm(e)$ properly color $V(G)$.
\end{conjecture}
Conjecture~\ref{123list} is open even if we allow larger lists of some fixed size $k$. For Conjecture~\ref{12list}, the best result due to Zhu and Wong in~\cite{ZW} generalizes Theorem~\ref{tot123}.
\begin{theorem}\label{tot123list}
For every graph $G$, and for every assignment of lists of size $2$ to the vertices and of size $3$ to the edges of $G$, there exists a weighting $\omm: V(G)\cup E(G)\to \R$ from the lists,
such that the induced total vertex weights $w(v):=\omm(v)+\sum_{e \ni v}\omm(e)$ properly color $V(G)$.
\end{theorem}
All these questions also make sense for hypergraphs when the graph $G$ in the statements is replaced by a hypergraph $H$.
Note that it is easy to construct totally irregular hypergraphs, so the irregularity strength of a hypergraph may actually be $1$ in certain cases. In this manuscript, we want to first consider Conjecture~\ref{123} for hypergraphs. 

To start, we have to decide what we mean by a proper vertex coloring of a hypergraph as there are differing notions. We will consider the weakest notion and call a hypergraph properly colored if it does not contain a monochromatic edge, i.e. an edge containing only vertices from one color class.

Next, we have to classify all hypergraphs which do not allow a vertex coloring edge weighting at all. What is the analogon of an isolated edge in the graph case? We will call a set of vertices of any cardinality {\em twin set} if the vertices in the set are contained in the exact same set of edges. With this notion, it is easy to verify that the only obstacle is an edge consisting of a twin set. In the absence of such edges, a vertex coloring edge weighting with integer weights is always possible. So we will ask for such graphs, what is the minimum maximum edge weight we have to use?

Going from graphs to hypergraphs, one discovers several important classes of hypergraphs invisible in the graph case, we will consider three special classes. A hypergraph is called {\em $k$-uniform} if all its edges have size $k$. If any two edges in a hypergraph intersect in at most one vertex, we call the hypergraph {\em linear} (this property is also called {\em simple} in other places). Note that graphs are exactly the $2$-uniform linear hypergraphs.
A hypergraph is called {\em bipartite} if it allows a proper $2$-coloring. 

In general, we allow multiple edges in our hypergraphs to make some operations easier to state. 
%
Starting with a hypergraph $H$ with vertex set $V(H)$ and edge set $E(H)$ and a vertex $v\in V(H)$, we define the hypergraph $H-v$ (the deletion of $v$) as the hypergraph with
\begin{align*}
V(H-v)&=V(H)\setminus\{v\},\\
E(H-v)&=\{e\setminus\{v\}:e\in E(H)\}.
\end{align*}
In other words, we delete $v$ from every edge, and we keep the resulting edges.

On the other hand, for $X\subseteq V(H)$, we may consider the {\em induced} hypergraph $H[X]$ with
 \begin{align*}
V(H[X])&=X,\\
E(H[X])&=\{e\in E(H):e\subseteq X\}.
\end{align*}
This time, we only allow edges completely contained in the smaller vertex set.

In the next section we provide some hypergraphs giving lower bounds for a number replacing the $3$ in the 1-2-3-Conjecture. In particular, we show that the statement of the 1-2-3-Conjecture can not be true for general hypergraphs. In fact, it would fail even for linear bipartite hypergraphs.

In the third section, we present the main result of the paper---an upper bound for the hypergraph version of the 1-2-3-Conjecture. We will get a bound for hypergraphs depending linearly on the size $r$ of the largest edge, which matches our lower bound as long as $r\ge 4$. 
In Section~\ref{total}, we will turn to Conjecture~\ref{12con}, and extend Theorem~\ref{tot123} to hypergraphs.


\section{Lower Bounds}
Let $F$ be any hypergraph with vertex set $V(F)$ and edge set $E(F)$, and suppose that $F$ has minimum degree at least $2$ and no edges of size smaller than $2$. From this, we create another hypergraph $H$, the dual of the incidence graph of $F$, as follows. Let $V(H)$ consist of the vertex-edge incidences in $F$, i.e., pairs $(v,e)$ where $v\in V(F)$, $e\in E(F)$ and $v\in e$. Let 
\begin{align*}
E_1(H)&= \bigcup_{v\in V(F)}\{\{(v,e)\in V(H):e\in E(F)\}\},\\
E_2(H)&= \bigcup_{e\in E(F)}\{\{(v,e)\in V(H):v\in V(F)\}\},\\
E(H)&=E_1(H)\cup E_2(H).
\end{align*}
With this construction, $H$ is linear and $2$-regular, the largest edge in $E_1(H)$ has size equal to the maximum degree in $F$, and the largest edge in $E_2(H)$ has size equal to the largest edge in $F$. 
%
Further, $H$ is bipartite as we can properly $2$-color $H$ as follows. Let $G$ be the graph underlying $H$, i.e. $V(G)=V(H)$ and $(v,e)(v',e')\in E(G)$ if and only if $|\{v,e\}\cap\{v',e'\}|=1$. As $H$ is linear, we can greedily find a subgraph $G'\subseteq G$ such that every edge in $E(H)$ contains exactly one edge in $G'$. As $H$ is $2$-regular, $G'$ has maximum degree at most $2$. Further, the edges of $G'$ naturally partition into two matchings $E_1(G')$ and $E_2(G')$ by their association with $E_1(H)$ and $E_2(H)$. Therefore, $G'$ has no odd cycles, and we can find a proper 2-vertex-coloring of $G'$, which is also a proper 2-vertex-coloring of $H$ by construction.

Further, suppose that $F$ has chromatic number $\chi(F)$, and consider any weighting $\omm: E(H)\to \{1,2,\ldots,\chi(F)-1\}$.  
Note that every edge $h_1=\{(v,e_1),(v,e_2),\ldots,(v,e_r)\}\in E_1(H)$ corresponds to the vertex $v\in V(F)$, and every edge 
$h_2=\{(v_1,e),(v_2,e),\ldots,(v_s,e)\}\in E_2(H)$ corresponds to the edge $e\in E(F)$. With this correspondence, we can define 
$\omm(v):=\omm(h_1)$ and $\omm(e):=\omm(h_2)$, and then
\[
\omm((v,e))=\sum_{h\ni (v,e)}\omm(h)=\omm(v)+\omm(e).
\]
Thus, the edge $h_2$ from above is monochromatic if and only if all the $\omm(v_i)$ are the same. If $\omm$ induces a proper coloring on the vertices of $H$, then $\omm$ has to be a proper coloring on the vertices of $F$, a contradiction. Therefore, $\omm$ cannot induce a proper vertex coloring of $H$.

This construction gives us several lower bounds. If we start with a complete graph $F$ on $r+1$ vertices, we obtain a hypergraph with maximum edge size $r$,  which needs a weight set of at least $\{1,2,\ldots,r+1\}$ on the edges to properly color the vertices. Starting with any other $r$-regular graph $F$ with chromatic number $r$,  we obtain a hypergraph with maximum edge size $r$,  which needs a weight set of at least $\{1,2,\ldots,r\}$ on the edges to properly color the vertices.

If we start with the Fano plane (or any other $3$-regular $3$-uniform non-bipartite hypergraph), we obtain a $2$-regular $3$-uniform hypergraph, which needs a weight set of at least $\{1,2,3\}$ on the edges to properly color the vertices. 

On the other hand, this construction cannot give us non-trivial examples for $r$-uniform hypergraphs with $r\ge 4$.  Thomassen shows in~\cite{T} that all $r$-uniform $r$-regular hypergraphs are bipartite for $r\ge 4$, leaving open the possibility of a vertex coloring edge weighting from the set $\{1,2\}$ for such graphs.

\section{Upper Bounds}

For $r\ge 4$, we show that the bound of $r+1$ we provided in the last section is in fact best possible. 
\begin{theorem}\label{Thm:Main}
For every hypergraph $H$ with edges of order at most $r\ge 2$, and no edge consisting of a twin set, there is a weighting $\omm: E(H)\to \{1,2,\ldots,\max\{5,r+1\}\}$,
such that the induced vertex weights $\omm(v):=\sum_{e\ni v}\omm(e)$ properly color $V(H)$.
\end{theorem}

\begin{proof}
We prove the statement by induction on $n=|V(H)|$. 
In fact, we will prove a slightly stronger statement to make the induction work. The statement is stronger because we can pick $\omm'$ to be constant:\\

{\em 
For every hypergraph $H$ with all edges of order between $2$ and $r\ge 2$, and no edge consisting of a twin set, and for every weighting of the vertices $\omm ':V(H)\to \N$, there is a weighting $\omm: E(H)\to \{1,2,\ldots,\max\{5,r+1\}\}$,
such that the induced vertex weights $\omm(v):=\omm '(v)+\sum_{e\ni v}\omm(e)$ properly color $V(H)$.
}\\

The statement is easy for $n=3$, so assume that $n\ge 4$.
We may assume that there are no twin sets of size larger than 1, otherwise we can remove all but one of the vertices of such a twin set and use induction. Further, we may assume that there are no multiple edges, as we can fix a weight on one of them and forget about the edge in the further argument.
We may also assume that every vertex lies in an edge of size $2$. Otherwise, pick a vertex $v$ which is in no edge of size $2$, and consider $H-v$. The only edges changed by this are edges containing $v$. If one of these edges now consists of a twin set, then that twin set was a twin set of size larger than $1$ in $H$, a contradiction.
Any edge weighting inducing a proper coloring on $H-v$ then induces a proper coloring on $H$ as well. 

The main idea of the proof is as follows. We order the vertices in a specific ordering and then only consider an associated graph on the same vertex set, and  all edges consisting of the first two vertices in each edge of $H$. Then we proceed very similarly to the proof in~\cite{KKP} to weight these edges respecting the order of the vertices, guaranteeing that in the end, the first two vertices of each edge in $H$ have different weighted degrees. Minor shortcomings we can fix in the end.
The main difficulty in this approach lies in the fact that in the proof for graphs we must pick a vertex ordering with specific properties, but as the associated graph of the hypergraph depends on the ordering of the vertices, we cannot change this vertex ordering after restricting our attention to the associated graph. To circumvent this problem,
we very carefully pick the ordering such that the resulting associated graph already has properties very close to what we need to make the graph process work without reordering.

In this spirit, define $E_2$ to be the set of edges of size $2$.
Let $E_1\subseteq E(H)$ be the set of edges in $H$ which contain an edge of $E_2$, and let $E_0=E(H)\setminus (E_1\cup E_2)$. 
For any ordering $\pi$ of the vertices, let $E_\pi$ be the set of pairs of vertices appearing first and second with respect to $\pi$ in an edge of size at least $3$ in $H$. Now let us find a suitable ordering $\pi$, building backwards from the last vertex in the ordering.

If $H$ contains a vertex incident to at least two edges in $E_2$, make such a vertex the last vertex $v_n$, 
and pick $v_{n-1},v_{n-2}$ such that $v_{n-1}v,v_{n-2}v\in E_2$. Let $h=v_{n-2}v$, $t=2$,
and skip forward to the next paragraph of this construction. 
Now assume that $H$ contains no such vertex, so $E_2$ is a perfect matching. 
If $E_0=\emptyset$, then the proper weighting is easy. Changing the weight of an edge in $E_1$ does not impact the properness of the coloring of the contained edge in $E_2$. Since $H$ has no edge consisting of a twin set, every edge in $E_2$ has an intersection of exactly one vertex with at least one edge in $E_1$.
Note that every edge in $E_1$ can intersect at most $r-1$ edges in $E_2$, so we can greedily pick a weight for each edge in $E_1$ such that all edges in $E_2$ intersecting that edge in exactly one vertex are properly colored.
Since all edges in $E_1$ contain edges in $E_2$, this in turn makes the coloring of every edge in $E_1$ proper.
Thus, we may assume that $E_0\ne \emptyset$. 
Find a hyperedge $h\in E_0$ of minimal size $3\le t\le \frac{n}2$, and choose $\pi$ such that $h=\{v_n,v_{n-2},\ldots,v_{n-2t+2}\}$, and $v_{n-2i}v_{n-2i-1}\in E_2$ for $0\le i\le t-1$.

Then, successively for $i\ge 2t$, let $v_{n-i}$ be a vertex not in the set of previously selected vertices $\{v_{n-i+1},\ldots,v_n\}$ connected to $\{v_{n-i+1},\ldots,v_n\}$ by an edge in $E_2\cup E_\pi$, as long as  such a vertex exists. Note that at this point, $\pi$ is determined sufficiently to decide if there is an edge in $E_\pi$ between $\{v_{n-i+1},\ldots,v_n\}$ and the remaining vertices. If we arrive at $v_1$ this way, $\pi$ is determined. If the process stops before, say after assigning $i$ labels, delete the previously ordered vertices from $H$ to form a hypergraph $H'=H-\{v_{n-i+1},\ldots,v_n\}$ on $n'=n-i$ vertices. Clearly, $H'$ contains no edges of size $1$. As $H$ contains no twin set of size $2$, $H'$ contains no twin set of size $2$ either, and therefore no edges consisting of twin sets. By induction, we can find a vertex coloring edge weighting on $H'$. Similarly, let $H''=H[\{v_{n-i+1},\ldots,v_n\}]$ be the connected hypergraph induced by $H$ on the previously selected vertices. Add the weights of edges intersecting both $V(H')$ and $V(H'')$ which we computed in the weighting of $H'$ to the respective vertex weights in $H''$, and use induction to weight the edges in $H''$, finishing the proof. Thus, we may assume in the following that $\pi$ is completely determined, and let 
 $G=G_\pi$  be the graph with edge set $E_2\cup E_\pi$.

Notice that $G$  is ordered in a way that every vertex up to $v_{n-2t+3}$ has a neighbor later in the order.
 
 Now we assign weights very similarly to the proof in~\cite{KKP}, we repeat large parts of this proof here so that this article is self contained. When we assign a weight to an edge in $G$, we are assigning it at the same time to the edge in $H$ that corresponds to the edge in $G$. 



We start by assigning the provisional weight $\om(e)=3$ to every edge $e$ in $G$ and adjust it at most twice while going through the vertices in the order specified by $\pi$---once when we are considering the first vertex in the edge, and once when we consider the second vertex. To every vertex $v_i$ with $i<n-2t+2$, we will assign a set of two colors $W(v_i)=\{w(v_i), w(v_i)+2\}$ with $w(v_i)\in \{ 0,1\}\mod 4$, so that for every edge $v_jv_i\in E(G)$ with $1\le j<i$, we have $W(v_j)\cap W(v_i)=\emptyset$, and we will guarantee that $\om(v_i)=\om'(v_i)+\sum_{e\ni v_i}\om(e)\in W(v_i)$. 

To this end, let $\om(v_1)=\omm '(v_1)+3d_H(v_1)$, and pick the set $W(v_1)=\{ w(v_1),w(v_1)+2\}$ so that $\om(v_1)\in W(v_1)$ and $w(v_1)\in \{ 0,1\}\mod 4$.
Let $2\le k\le n-2t+1$ and assume that we have picked $W(v_i)$ for all $i<k$ and
\begin{itemize}
\item $\om(v_i)\in W(v_i)$ for $i<k$,
\item $\om(v_kv_j)=3$ for all edges in $G$ with $j>k$, and
\item if $\om(v_iv_k)\ne 3$ for some edge in $G$ with $i<k$, then $\om(v_iv_k)=2$ and $\om(v_i)=w(v_i)$ or $\om(v_iv_k)=4$ and $\om(v_i)=w(v_i)+2$.
\end{itemize}
%

If $v_iv_k\in E(G)$ for some $i<k$ we can either add or subtract $2$ to $\om(v_iv_k)$ keeping $\om(v_i)\in W(v_i)$. If $v_k$ has $d$ such neighbors, this gives us a total of $d+1$ choices (all of the same parity) for $\om(v_k)$. In addition to this we will allow to alter the weight $\om(v_kv_j)$ by $1$, where $j>k$ is smallest such that $v_kv_j\in E(G)$. This way, $\om(v_k)$ can take all integer values of an interval $[a,a+2d+2]$. We want to adjust the weights and assign $w(v_k)$ so that
\begin{enumerate}
\item $\om(v_i)\in W(v_i)$ for $1\le i\le k$,
\item $w(v_i)\ne w(v_k)$ for $v_iv_k\in E(G)$ with $i<k$, and\label{con1}
\item either $\om(v_k)=w(v_k)$ and $\om(v_kv_j)\in \{2,3\}$ or $\om(v_k)=w(v_k)+2$ and $\om(v_kv_j)\in \{3,4\}$.\label{con2}
\end{enumerate}
Condition~(\ref{con1}) can block at most $2d$ values in $[a,a+2d+2]$, and condition~(\ref{con2}) can block only the values $a$ and $a+2d+2$ (for all other values $\om(v_k)$ with $\om(v_kv_j)\ne 3$, we have the choice between $\om(v_kv_j)=2$ and $\om(v_kv_j)=4$). At least one value remains open for $\om(v_k)$.

This way, we can assign the sets $W(v_k)$ step by step for all $k\le n-2t+1$ without conflicts. Note that the first time $\om(v_k)$ may get changed by an adjustment of an edge $v_kv_i$ for $i>k$ is when $i=j$, so we don't run into problems with edges weighted $2$ or $4$.

%

As the final step, we make sure that the remaining edges are also colored properly by the weighting. 
We consider the cases $t=2$ and $t\ge 3$ separately, starting with $t=2$.
In this case, continue the previous process until $v_{n-1}$ in $G$, which is possible as both $v_{n-2}$ and $v_{n-1}$ are adjacent to $v_n$. Now we have to find an open weight for $v_n$.
If $v_iv_n\in E(G)$ for some $i<n$ we can again either add or subtract $2$ to $\om(v_iv_n)$ keeping $\om(v_i)\in W(v_i)$.
These possible adjustments give a total of $d_{G}(v_n)+1\ge 3$ options (all of the same parity) for $\om(v_n)$. Hence if the smallest such
option $a$ has $a\in\{2,3\}\mod 4$, then picking the lower possible weight on each edge incident to $v_n$ gives a proper coloring of the vertices. If $a\in\{0,1\}\mod 4$ and there is a $v_i\in N(v_n)$ with $w(v_i)\ne a$, then picking the higher weight on $v_iv_n$ and the lower weight on all other edges gives $\om(v_n)=a+2$ in a proper coloring. Finally, if $a\in\{0,1\}\mod 4$ and $w(v_i)= a$ for all $v_i\in N(v_n)$, picking the higher weight on at least two edges gives a proper coloring.

For $t\ge 3$, observe that after dealing with $v_{n-2t+1}$, every edge in $H$ which contains at least two vertices in $\{v_1,v_2,\ldots,v_{n-2t+1}\}$ is properly colored, and we will not change the weights of these vertices in the remainder of the proof. Every remaining edge in $E_1$ contains at least one of the pairs $\{v_{n-2i},v_{n-2i-1}\}$ for $0\le i\le t-1$. Every remaining edge in $E_0$ contains at least $t-1$ vertices in $\{v_{n-2t+2},\ldots,v_n\}$ by the minimality of $t=|h|$, and therefore at least one of the pairs $\{v_{n-2t+2},v_{n-3}\}$, $\{v_{n-2t+2},v_{n-2}\}$,  $\{v_{n-2t+2},v_{n-1}\}$, $\{v_{n-2t+2},v_{n}\}$, $\{v_{n-3},v_{n-1}\}$, $\{v_{n-3},v_{n}\}$, $\{v_{n-2},v_{n-1}\}$,  $\{v_{n-2},v_{n}\}$. Therefore, if we adjust the weights of the edges $h$, $v_{n-3}v_{n-2}$ and $v_{n-1}v_n$ such that all these $t+8$ pairs receive different weights, our weighting is complete.

First, adjust $h$ such that all of the pairs $\{v_{n-2i},v_{n-2i-1}\}$ for $0\le i\le t-1$ are colored properly. As we have at least $r+1\ge t+1$ weights to choose from, this is possible. 
From each of the at least five choices for  $\om(v_{n-3}v_{n-2})$ and $\om(v_{n-1}v_{n})$, at least three choices will guarantee that $\om(v_{n-2t+2})\ne \om(v_{n-i})$ for $0\le i\le 3$. So let $x_1<x_2<x_3$ be such remaining choices for $\om(v_{n-3}v_{n-2})$, and let $y_1<y_2<y_3$ be such remaining choices for $\om(v_{n-1}v_{n})$.
Notice that $y_1-x_3<y_1-x_2<y_1-x_3<y_2-x_3<y_3-x_3$, so we have at least five different options for $\om(v_{n-1}v_{n})-\om(v_{n-3}v_{n-2})$ at this point. Therefore, at least one of these options will yield a result with $\om(v_{n-i})\ne \om(v_{n-j})$ for all $0\le i\le 1$, $2\le j\le 3$. This proves the theorem.
\end{proof}

\section{Total Colorings}\label{total}
In this short section, we generalize Theorem~\ref{tot123}  
to hypergraphs. While a generalization of the stronger Theorem~\ref{tot123list} to hypergraphs would be very interesting, we are not able to prove such a statement at this point.

\begin{theorem}
For every hypergraph $H$ without edges of size $0$ or $1$, there are weightings $\omm: E(H)\to \{1,2,3\}$ and $\omm': V(H)\to \{1,2\}$
such that the induced total vertex weights $w(v):=\omm'(v)+\sum_{e\ni v}\omm(e)$ properly color $V(H)$.
\end{theorem}

\begin{proof}
In comparison with the previous theorem, the proof of Theorem~\ref{tot123} generalizes to hypergraphs with even fewer complications. To this end, order the vertices of $H$ in any order $\pi$, and consider the multigraph $G=G_{\pi}$ as above, i.e. the graph on the same vertex set with edges $u_ev_e$, where $u_e$ and $v_e$ are the first and second vertex appearing in some edge $e\in E(H)$. We will show the stronger statement that we can find a weighting such that $w(u_e)\ne w(v_e)$ for every edge $e\in E(H)$.

Start with all edges weighted $2$, and all vertices weighted $1$ to give a temporary total weight to every vertex.
Change the weights of edges and vertices greedily going through the vertices in the order given by $\pi$. When we consider a vertex $v$, the only edge weights we may adjust are those of edges $e$ with $v_e=v$, and the vertex weights of $u_e$ and $v_e$ in these edges.
Note that we can adjust the weight of every such edge $e$ by either $+1$ or $-1$, and change the weight of $u_e$ at the same time to keep the total weight of $u_e$ unchanged, always using only allowable weights.
This gives us sufficiently many options for the total weight of $v$ to make sure that this total weight is different from all the total weights of  the vertices $u_e$ with $v_e=v$ (if there are $d$ such vertices, we have $d+1$ different options for the total weight of $v$).
As the total weight of $v$ is fixed for the remainder of the process, we end up with a weighting with the desired properties.
\end{proof}

\section{Conclusion and Open Questions}

The only class of hypergraphs we know achieving the bound in Theorem~\ref{Thm:Main} is the one constructed above, stemming from the complete graphs $K_{r+1}$. Possibly, it is true that this is the unique example for $r\ge 3$, and in all other cases a set $\{1,2,\ldots,r\}$ is sufficient.

Note that most of our examples on the lower bounds are highly non-uniform, they contain very small and very large edges. 
For $r$-uniform hypergraphs, there may be a constant upper bound instead, independent of $r$. For $r=3$, we conjecture the following.
\begin{conjecture}\label{con3}
For every $3$-uniform hypergraph $H$ without an isolated edge, there is a weighting $\omm: E(H)\to \{1,2,3\}$,
such that the induced vertex weights $\omm(v):=\sum_{e\ni v}\omm(e)$ properly color $V(H)$.
\end{conjecture}
At some point we thought that it may even be true that for $r\ge 4$, the set $\{1,2\}$ is sufficient, and in fact asked this question in a previous version of this manuscript. 
Then Bennett, Dudek, Frieze and Helenius~\cite{BenDFH} constructed a family of  uniform hypergraphs for all $r$ which require the set $\{1,2,3\}$, and extended Conjecture~\ref{con3} accordingly. 
\begin{conjecture}\cite{BenDFH}
For every $r$-uniform hypergraph $H$ without an isolated edge, there is a weighting $\omm: E(H)\to \{1,2,3\}$,
such that the induced vertex weights $\omm(v):=\sum_{e\ni v}\omm(e)$ properly color $V(H)$.
\end{conjecture}
The construction in~\cite{BenDFH} is as follows. Start with a cycle of length $4k+2$, and transform every $2$-edge into an $r$-edge by producing $(r-2)$ twins of every other vertex in the cycle. As the cycle cannot be properly colored by an edge weighting from $\{1,2\}$, this is impossible for this hypergraph as well. If we forbid this type of construction by forbidding twins, we think the following is true.
\begin{conjecture}
For large enough $r$, for every $r$-uniform hypergraph $H$ without twins, there is a weighting $\omm: E(H)\to \{1,2\}$,
such that the induced vertex weights $\omm(v):=\sum_{e\ni v}\omm(e)$ properly color $V(H)$.
\end{conjecture}



%

 \bibliographystyle{amsplain}

\end{document}